\documentclass[12pt]{amsart}
\usepackage{framed,afterpage}
\usepackage[margin=1in,letterpaper]{geometry}
\usepackage{amssymb,amsmath,amsthm}
\usepackage{tikz}
\usepackage{hyperref}
\usepackage{dsfont}
\hypersetup{colorlinks, citecolor=red, filecolor=black, linkcolor=blue, urlcolor=blue}

\newtheorem{theorem}{Theorem}[section]
\newtheorem{corollary}[theorem]{Corollary}
\newtheorem{lemma}[theorem]{Lemma}
\newtheorem{proposition}[theorem]{Proposition}
\newtheorem{question}[theorem]{Question}

\theoremstyle{definition}
\newtheorem{definition}[theorem]{Definition}
\newtheorem{example}[theorem]{Example}

\newcommand{\Nn}{\mathbb{N}}

\newcommand{\Rr}{\mathbb{R}}

\newcommand{\aaa}{\mathbf{a}}
\newcommand{\bbb}{\mathbf{b}}
\newcommand{\ccc}{\mathbf{c}}

\newcommand{\kkk}{\mathbf{k}}
\newcommand{\bl}{\ensuremath{\boldsymbol\ell}}

\newcommand{\ppp}{\mathbf{p}}
\newcommand{\qqq}{\mathbf{q}}

\newcommand{\sss}{\mathbf{s}}
\newcommand{\ttt}{\mathbf{t}}
\newcommand{\uuu}{\mathbf{u}}
\newcommand{\xxx}{\mathbf{x}}
\newcommand{\yyy}{\mathbf{y}}

\newcommand{\compn}{\vDash}
\newcommand{\sm}{\setminus}

\DeclareMathOperator{\Shuffle}{Shuff}

\DeclareMathOperator{\JOIN}{\mathbf{Join}}
\DeclareMathOperator{\SPLIT}{\mathbf{Split}}

\newcommand{\concat}{\bullet}
\newcommand{\x}{\times}

\colorlet{darkgreen}{green!40!black}

\colorlet{darkblue}{blue!50!black}

\newcommand{\AAA}{\mathbf{A}}

\renewcommand{\Join}{\textbf{Join}}
\newcommand{\Split}{\textbf{Split}}

\newcommand{\PF}{\mathsf{PF}} % The basic symbol for parking functions, for use in other macros
\newcommand{\PC}{\mathsf{PC}} % Ditto for parking completions
\newcommand{\IPF}{\mathsf{IPF}} % Ditto for increasing parking completions
\newcommand{\IPC}{\mathsf{IPC}} % Ditto for increasing parking completions
\newcommand{\pco}[2]{\PC_{#1}(#2)} % Parking completions (#1 = number of cars, #2 = partial parking function = locations of trailers)
\newcommand{\ipco}[2]{\IPC_{#1}(#2)} % Increasing Parking completions (#1 = number of cars, #2 = partial parking function = locations of trailers)

\newcommand{\excise}[1]{}  % for quickly commenting out large blocks of text
\title[Enumerating parking completions using Join and Split]{Enumerating Parking Completions Using Join and Split}
\thanks{This work was completed in part at the 2018 Graduate Research Workshop in Combinatorics, which was supported in part by NSF grants \#1604458 and \#1603823, NSA grant \#H98230-18-1-0017, a generous award from the Combinatorics Foundation, and Simons Foundation Collaboration Grants \#426971 (to M. Ferrara) and \#316262 (to S. Hartke).\\ 
\indent JLM was supported in part by Simons Foundation Collaboration Grant \#315347.\\
\indent SB was supported in part by Simons Foundation Collaboration Grant \#427264.}
\makeatletter
\let\thetitle\@title        % Document title saved in command
\let\theauthor\@author      % Document author saved in command

\date{\today}

\author[Adeniran]{Ayomikun Adeniran}
\address[A. Adeniran]{Department of Mathematics, Texas A\&M University, United States}
\email{\href{mailto:ayoijeng88@tamu.edu}{ayoijeng88@tamu.edu}}

\author[Butler]{Steve Butler}
\address[S. Butler]{Department of Mathematics, Iowa State University, United States}
\email{\href{mailto:butler@iastate.edu}{butler@iastate.edu}}

\author[Dorpalen-Barry]{Galen Dorpalen-Barry}
\address[G. Dorpalen-Barry]{Department of Mathematics, University of Minnesota, United States}
\email{\href{mailto:dorpa003@umn.edu}{dorpa003@umn.edu}}

\author[Harris]{Pamela E. Harris}
\address[P. E. Harris]{Department of Mathematics and Statistics, Williams College, United States}
\email{\href{mailto:peh2@williams.edu}{peh2@williams.edu}}

\author[Hettle]{Cyrus Hettle}
\address[C. Hettle]{School of Mathematics, Georgia Institute of Technology, United States}
\email{\href{mailto:chettle@gatech.edu}{chettle@gatech.edu}}

\author[Liang]{Qingzhong Liang}
\address[Q. Liang]{Department of Mathematics, Duke University, United States}
\email{\href{mailto:qingzhong.liang@duke.edu}{qingzhong.liang@duke.edu}}

\author[Martin]{Jeremy L. Martin}
\address[J. L. Martin]{Department of Mathematics, University of Kansas, United States}
\email{\href{mailto:jlmartin@ku.edu}{jlmartin@ku.edu}}

\author[Nam]{Hayan Nam}
\address[H. Nam]{Department of Mathematics, Iowa State University, United States}
\email{\href{mailto:hnam@iastate.edu}{hnam@iastate.edu}}

\begin{document}

\begin{abstract}
Given a strictly increasing sequence $\ttt$ with entries from $[n]:=\{1,\ldots,n\}$, a \emph{parking completion} is a sequence $\ccc$ with $|\ttt|+|\ccc|=n$ and $|\{t\in \ttt\mid t\le i\}|+|\{c\in \ccc\mid c\le i\}|\ge i$ for all $i$ in $[n]$.  We can think of $\ttt$ as a list of spots already taken in a street with $n$ parking spots and $\ccc$ as a list of parking preferences where the $i$-th car attempts to park in the $c_i$-th spot and if not available then proceeds up the street to find the next available spot, if any.  
A parking completion corresponds to a set of preferences $\ccc$ where all cars park.

We relate parking completions to enumerating restricted lattice paths and give formulas for both the ordered and unordered variations of the problem by use of a pair of operations termed \textbf{Join} and \textbf{Split}.  Our results give a new volume formula for most Pitman-Stanley polytopes, and enumerate the \emph{signature parking functions} of Ceballos and Gonz\'alez D'Le\'on.
\end{abstract}
\maketitle

\section{Introduction}
Let $[n]=\{1,\dots,n\}$.  A \emph{parking function} is a sequence $\ccc=(c_1,\dots,c_n)\in[n]^n$ such that there exists a permutation $\sigma$ of $[n]$ such that $c_{\sigma(i)} \leq i$ for all $i\in[n]$.   The name ``parking function'' comes from the following setup: $n$ cars attempt to park in a one-way street with spots labeled from $1$ to $n$.  The $i$-th car parks at its preferred spot $c_i$, if it is available; otherwise, it parks at the first available spot after $c_i$.  If all spots from $c_i$ on are filled, then the $i$-th car cannot park.  A parking function is then a sequence of preferences so that all $n$ cars are able to park.  It is well known that the number of parking functions of length~$n$ is $(n+1)^{n-1}$, and the number of \emph{increasing parking functions} (i.e., increasing rearrangements of parking functions) is the Catalan number $C_n=\frac{1}{n+1}\binom{2n}{n}$.  In addition, both parking functions and increasing parking functions can be interpreted as labeled and unlabeled lattice paths, respectively.  For more on parking functions and their extensive combinatorial connections, see the survey by Yan \cite{yan}.  

Suppose that $m$ of the $n$ spots are already taken.  We want to determine the possible preferences for $n-m$ cars so that they can all successfully park.  We call the set of successful preference sequences the \emph{parking completions for a sequence $\ttt=(t_1,\dots,t_m)$} where the entries of $\ttt$ denote which spots are \emph{taken} in increasing order.  We will let $\pco{n}{\ttt}$ denote the set of parking completions of $\ttt$ in $[n]$ and $\ipco{n}{\ttt}$ denote the set of (weakly) increasing parking completions of $\ttt$ in $[n]$.

\begin{theorem}\label{thm:main2}
The number of parking completions of $\ttt=(t_1,\dots,t_m)$ in $[n]$ is 
\[
\big|\pco{n}{\ttt}\big|=\sum_{\bl\in L_n(\ttt)} \binom{n-m}{\bl}\prod_{j=1}^{m+1}(\ell_j+1)^{(\ell_j-1)},
\]
where
\[L_n(\ttt)= \left\{\bl=(\ell_1,\dots,\ell_{m+1})\in\Nn^{m+1}\,\middle\vert
\begin{array}{l}
\ell_1+\cdots+\ell_j \geq t_j-j \text{ for all } j\in[m], \text{ and}\\
 \ell_1+\cdots+\ell_{m+1}=n-m 
 \end{array} \right\}.\]
\end{theorem}

The notion of parking completions unifies work of Yan~\cite{yan-GPF}, Gessel--Seo \cite{GesselSeo} (under the name of \emph{$c$-parking functions}), and Ehrenborg--Happ~\cite{ehrenborg1,ehrenborg2}, who studied the case $\ttt=(1,\dots,\ell)$, and Diaconis--Hicks \cite{diaconis}, who considered the case that $\ttt=(i)$ for some $i\in[n]$.
Although it is not immediate, Theorem \ref{thm:main2} implies the following generalization of Diaconis and Hicks' enumeration of parking function shuffles \cite[Corollary 1]{diaconis}.

\begin{corollary}\label{thm:main1}
Let $m,n\ge 1$ and let $1\le i\le n-m$.  Then 
\[
|\pco{n}{(i+1,\ldots,i+m)}\big|=
\sum_{k=i}^{n-m}\binom{n-m}{k}(k+1)^{k-1}m(n-k)^{n-k-m-1}.
\]
\end{corollary}

For increasing parking completions, there is a similar result.
\begin{theorem}\label{count-IPC}
The number of increasing parking completions of $\ttt=(t_1,\ldots,t_m)$ in $[n]$ is
\[
\big|\IPC_n(\ttt)\big| = \sum_{\ell\in L_n(\ttt)} \prod_{i=1}^{m+1} \frac{1}{\ell_i+1} \binom{2\ell_i}{\ell_i},
\]
where
\[
L_n(\ttt)= \left\{\bl=(\ell_1,\dots,\ell_{m+1})\in\Nn^{m+1}\,\middle\vert
\begin{array}{l}
\ell_1+\cdots+\ell_j \geq t_j-j \text{ for all } j\in[m], \text{ and}\\
 \ell_1+\cdots+\ell_{m+1}=n-m 
 \end{array} \right\}.
 \]
\end{theorem}

For the special case $\ttt=()$ (where no spots are taken), these formulas become $\big|\PC_n(\ttt)\big|=(n+1)^{n-1}$ and $\big|\IPC_n(\ttt)\big|=\frac1{n+1}\binom{2n}n$, the classical counts for parking functions and increasing parking functions, respectively.

Parking completions turn out to be a mild specialization of what were called \emph{generalized $\xxx$-parking functions} in~\cite{yan-GPF}.  Pitman and Stanley~\cite[Theorem~1]{PitmanStanley} showed that the number of $\xxx$-parking functions equals the volume of what is now called a \textit{Pitman-Stanley polytope}, and gave a formula which our Theorem~\ref{thm:main1} superficially resembles.  In fact the two formulas are not equivalent, often producing incomparable expressions for the number of parking completions as a composition (see Example~\ref{compare-PS-and-our-formula}), so that Theorem~\ref{thm:main1} provides a new volume formula for (most) Pitman-Stanley polytopes,

\textit{Rational Catalan combinatorics} is a recent development that has attracted significant attention (see \cite{ALW,ARW,GMV1,GMV2}).  The essential idea is to replace the condition $c_i\leq i$ in the definition of an increasing parking function with $c_i\leq qi$, where $q$ is a rational number.  Ceballos and Gonz\'alez~D'Le\'{o}n~\cite{CeballosGonzalez} studied a further generalization, \emph{signature Catalan combinatorics}.  Signature parking functions correspond to labeled lattice paths constrained to lie in some fixed Ferrers diagram.  There is in fact an easy bijection (Proposition~\ref{dyck-signature-equivalence}) between increasing parking completions and the signature Dyck paths of~\cite{CeballosGonzalez}, raising the possibility of applying our results to signature and rational Catalan combinatorics.

The paper will proceed as follows.  We begin with background on parking functions, lattice paths, and tuples in Section~\ref{sec:preliminaries}.  In Section~\ref{sec:join-split}, we introduce two inverse bijections $\JOIN$ and $\SPLIT$ between lists of parking functions and arbitrary non-decreasing integer sequences.  These two maps are then used to prove the main results in Section~\ref{sec:enumeration-proofs}. In Section~\ref{sec:enumerative-connections}, we discuss the applications to Pitman--Stanley polytopes and to signature Catalan combinatorics.

\section{Preliminaries}\label{sec:preliminaries} In this section, we give the necessary definitions and conventions for the rest of the paper.

\subsection{Parking Completions and Lattice Paths }\label{subsec:lattice-path}
Let $\ttt=(t_1,\dots,t_m)$ be a (strictly) increasing list of elements from $[n]$ and let $\uuu=(u_1,\ldots,u_{n-m})$ be the list of elements of $[n]\sm\{t_1,\dots,t_m\}$, sorted in (strictly) increasing order.  If we think of $\ttt$ as the taken spots, then $\uuu$ can be interpreted as the unoccupied spots.

\begin{definition} \label{defn:parking-completion}
An \emph{increasing parking completion} of $\ttt$ is a weakly increasing sequence $(c_1,\ldots,c_{n-m})\in[n]^{n-m}$ such that $(t_1,\dots,t_m,c_1,\dots,c_{n-m})$ is a parking function.  Equivalently, the sequence satisfies $c_i\le u_i$ for all $i$.  The set of all increasing parking completions of~$\ttt$ will be denoted $\IPC_n(\ttt)$.

A \emph{parking completion} of $\ttt$ is a sequence $\ccc\in[n]^{n-m}$ where when the entries are rearranged into a weakly increasing order the result is an increasing parking completion.  The set of all parking completions of~$\ttt$ will be denoted $\PC_n(\ttt)$.
\end{definition}

It is immediate from the definition that parking completions of~$\ttt$ are stable under rearrangement.

We will explore parking completions and increasing parking completions using lattice paths.
A \emph{lattice path} $L$ from $(1,1)$ to $(p+1,q+1)$ is a sequence of $p$ right steps and $q$ up steps on the integer lattice $\mathbb{Z}^2$.  Lattice paths with $q$ up steps are in bijection with weakly increasing nonnegative integer sequences $\ccc=(c_1,\dots,c_q)$: for each $1\leq i\leq q$ there is an up step at $x=c_i$.  An example of a lattice path is shown in Figure~\ref{fig:figure1}.

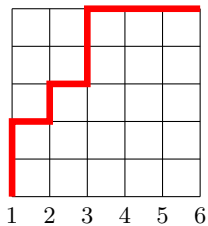
\begin{figure}[!htb]
\begin{center}
\begin{tikzpicture}[scale=0.5]
\foreach \x in {1,...,6} { \draw (\x,1)--(\x,6); \node at (\x,.5) {\scriptsize\x}; }
\foreach \y in {1,...,6} { \draw (1,\y)--(6,\y); }
\draw[line width = 2.5pt,red,shift={(1,1)}] (0,0)--(0,2)--(1,2)--(1,3)--(2,3)--(2,5)--(5,5);
\end{tikzpicture}
\end{center}
\caption{A lattice path with $p=q=5$ and $\ccc=(1,1,2,3,3)$.}
\label{fig:figure1}
\end{figure}

\begin{proposition}[{{\cite[Exercise 6.19(b),(s)]{stanley-ec2}}}]
There is a bijection between the lattice paths with $p=q=n$ that stay weakly above the diagonal (i.e., that contain no point $(i,j)$ with $i>j$) and increasing parking functions of length $n$.
\end{proposition}

We can modify this bijection to count increasing parking completions.  Suppose that $\ttt=(t_1,\dots,t_m)$ and $\uuu=(u_1,\dots,u_{n-m})$ are two increasing lists of integers whose disjoint union as sets is $[n]$.
Then
\begin{align*}
\IPC_n(\ttt)
&= \{(c_1,\dots,c_{n-m})\in[n]^{n-m} \mid  c_i\leq u_i \text{ for all $i$ and }c_1\le \cdots \le c_{n-m}\}.
\end{align*}
Consequently, we can describe $\IPC_n(\ttt)$ as the lattice paths lying \emph{weakly} above and to the left of the points $(i,u_i)$ for $1\leq i\leq n-m$.

Let $\lambda_{\uuu}=(u_1-1,\ldots,u_{n-m}-1)$ be an integer partition (where the terms are in \underline{increasing} order), and let $F_\uuu$ denote the corresponding Ferrers diagram (using the English convention).
Then 
\[
\big|\IPC_n(\ttt)\big|=\big|\{\text{lattice paths in $F_\uuu$}\}\big|.
\]
We will illustrate this in  Example~\ref{ex:ipc-det-small}.

We can use known results to count these lattice paths.  The Lindstr\"{o}m-Gessel-Viennot lattice path theorem (\cite[Theorem~2.7.1]{stanley-ec1}; see also \cite[Exercise 3.149]{stanley-ec1}) implies the following.

\begin{proposition}\label{cor:detform} 
Let $\ttt=(t_1,\dots,t_m)$ and $\uuu=(u_1,\dots,u_{n-m})$ be two increasing lists of integers whose disjoint union as sets is $[n]$.  Then
\[
\big|\ipco{n}{\ttt}\big| = \det\left[ \binom{u_{n-m+1-i}}{i-j+1} \right]_{i,j=1}^{n-m}.
\]
\end{proposition}

\begin{example}\label{ex:ipc-det-small} 
Let $n=4$, $m=2$, $\ttt=(1,3)$ and $\uuu=(2,4)$ (so $\lambda_\uuu=(1,3)$).  The lattice paths in $F_\uuu$ are shown in Figure~\ref{fig:figure2}. 

\begin{figure}[!ht]
\begin{center}
\begin{tikzpicture}[scale=0.5]
\fill[black!20!white] (1,1) rectangle (4,2) (1,0) rectangle (2,1);
\foreach \x in {1,...,4} { \draw (\x,0)--(\x,2); \node at (\x,-.5) {\scriptsize\x}; }
\foreach \y in {0,...,2} { \draw (1,\y)--(4,\y);}
\draw[line width = 2.5pt,red] (1,0)--(2,0)--(2,1)--(4,1)--(4,2);
\foreach \x/\y in {2/0, 4/1}
    \draw[blue,fill=blue] (\x,\y) circle (.2);
\node at (2.5,-1.5) {$(2,4)$};
\end{tikzpicture}\hfil
\begin{tikzpicture}[scale=0.5]
\fill[black!20!white] (1,1) rectangle (4,2) (1,0) rectangle (2,1);
\foreach \x in {1,...,4} { \draw (\x,0)--(\x,2); \node at (\x,-.5) {\scriptsize\x}; }
\foreach \y in {0,...,2} { \draw (1,\y)--(4,\y);}
\draw[line width = 2.5pt,red] (1,0)--(2,0)--(2,1)--(3,1)--(3,2)--(4,2);
\foreach \x/\y in {2/0, 4/1}
    \draw[blue,fill=blue] (\x,\y) circle (.2);
\node at (2.5,-1.5) {$(2,3)$};
\end{tikzpicture}\hfil
\begin{tikzpicture}[scale=0.5]
\fill[black!20!white] (1,1) rectangle (4,2) (1,0) rectangle (2,1);
\foreach \x in {1,...,4} { \draw (\x,0)--(\x,2); \node at (\x,-.5) {\scriptsize\x}; }
\foreach \y in {0,...,2} { \draw (1,\y)--(4,\y);}
\draw[line width = 2.5pt,red] (1,0)--(2,0)--(2,2)--(4,2);
\foreach \x/\y in {2/0, 4/1}
    \draw[blue,fill=blue] (\x,\y) circle (.2);
\node at (2.5,-1.5) {$(2,2)$};
\end{tikzpicture}\hfil
\begin{tikzpicture}[scale=0.5]
\fill[black!20!white] (1,1) rectangle (4,2) (1,0) rectangle (2,1);
\foreach \x in {1,...,4} { \draw (\x,0)--(\x,2); \node at (\x,-.5) {\scriptsize\x}; }
\foreach \y in {0,...,2} { \draw (1,\y)--(4,\y);}
\draw[line width = 2.5pt,red] (1,0)--(1,1)--(4,1)--(4,2);
\foreach \x/\y in {2/0, 4/1}
    \draw[blue,fill=blue] (\x,\y) circle (.2);
\node at (2.5,-1.5) {$(1,4)$};
\end{tikzpicture}\hfil
\begin{tikzpicture}[scale=0.5]
\fill[black!20!white] (1,1) rectangle (4,2) (1,0) rectangle (2,1);
\foreach \x in {1,...,4} { \draw (\x,0)--(\x,2); \node at (\x,-.5) {\scriptsize\x}; }
\foreach \y in {0,...,2} { \draw (1,\y)--(4,\y);}
\draw[line width = 2.5pt,red] (1,0)--(1,1)--(3,1)--(3,2)--(4,2);
\foreach \x/\y in {2/0, 4/1}
    \draw[blue,fill=blue] (\x,\y) circle (.2);
\node at (2.5,-1.5) {$(1,3)$};
\end{tikzpicture}\hfil
\begin{tikzpicture}[scale=0.5]
\fill[black!20!white] (1,1) rectangle (4,2) (1,0) rectangle (2,1);
\foreach \x in {1,...,4} { \draw (\x,0)--(\x,2); \node at (\x,-.5) {\scriptsize\x}; }
\foreach \y in {0,...,2} { \draw (1,\y)--(4,\y);}
\draw[line width = 2.5pt,red] (1,0)--(1,1)--(2,1)--(2,2)--(4,2);
\foreach \x/\y in {2/0, 4/1}
    \draw[blue,fill=blue] (\x,\y) circle (.2);
\node at (2.5,-1.5) {$(1,2)$};
\end{tikzpicture}\hfil
\begin{tikzpicture}[scale=0.5]
\fill[black!20!white] (1,1) rectangle (4,2) (1,0) rectangle (2,1);
\foreach \x in {1,...,4} { \draw (\x,0)--(\x,2); \node at (\x,-.5) {\scriptsize\x}; }
\foreach \y in {0,...,2} { \draw (1,\y)--(4,\y);}
\draw[line width = 2.5pt,red] (1,0)--(1,2)--(4,2);
\foreach \x/\y in {2/0, 4/1}
    \draw[blue,fill=blue] (\x,\y) circle (.2);
\node at (2.5,-1.5) {$(1,1)$};
\end{tikzpicture}
\end{center}
\caption{Lattice paths in $F_{(2,4)}$ and the corresponding increasing parking functions.}
\label{fig:figure2}
\end{figure}
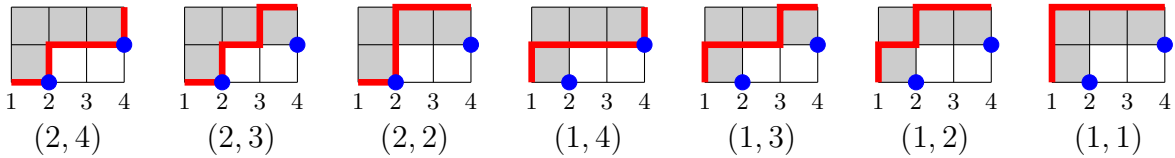

The computation from Proposition~\ref{cor:detform} gives:
\[
\big|\IPC_4((1,3))\big|=
\det\begin{pmatrix}
\binom{4}{1} & \binom{4}{0}\\[5pt]
\binom{2}{2} & \binom{2}{1} 
\end{pmatrix}
=\det\begin{pmatrix}
4 & 1 \\ 
1 & 2 
\end{pmatrix}
=7.
\]
\end{example}

\begin{example}\label{ex:unlabeled1} Let $n=8$, $m=4$, $\ttt=(1,2,6,7)$ and $\uuu=(3,4,5,8)$ (so $\lambda=(2,3,4,7)$).  Then the elements of $\IPC_8(\ttt)$ correspond to lattice paths from the bottom left to the top right of $F_\uuu$ shown in Figure~\ref{fig:figure3}.

\begin{figure}[!ht]
\begin{center}
\begin{tikzpicture}[scale=0.5]
\fill[color=black!20!white]
(1,0) rectangle (3,1)
(1,1) rectangle (4,2)
(1,2) rectangle (5,3)
(1,3) rectangle (8,4);
\foreach \x in {1,...,8} { \draw (\x,0)--(\x,4); \node at (\x,-.5) {\scriptsize\x}; }
\foreach \y in {0,...,4} { \draw (1,\y)--(8,\y); }
\foreach \x/\y in {3/0, 4/1, 5/2, 8/3} \draw[blue,fill=blue] (\x,\y) circle (.15);
\foreach \x in {1,...,8} \node at (\x,-.5) {\scriptsize\x};
\end{tikzpicture}
\end{center}
\caption{The sub-lattice $F_\uuu$ for Example~\ref{ex:unlabeled1}.}
\label{fig:figure3}
\end{figure}
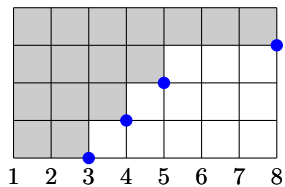

By Proposition~\ref{cor:detform},
\[
\big|\IPC_{8}((1,2,6,7))\big| = 
\det {\begin{pmatrix}
\binom{8}{1}&\binom{8}{0}&\binom{8}{-1}&\binom{8}{-2} \\[5pt]
\binom{5}{2}&\binom{5}{1}&\binom{5}{0}&\binom{5}{-1} \\[5pt]
\binom{4}{3}&\binom{4}{2}&\binom{4}{1}&\binom{4}{0} \\[5pt]
\binom{3}{4}&\binom{3}{3}&\binom{3}{2}&\binom{3}{1} \end{pmatrix}}
=
\det {\begin{pmatrix}
8 & 1 & 0 & 0\\
10 & 5 & 1 & 0\\
4 & 6 & 4 & 1\\
0 & 1 & 3 & 3
\end{pmatrix}}
= 146.
\]
\end{example}

We can extend the bijection between lattice paths and increasing parking completions to decorated lattice paths $L$ and general (not necessarily increasing) parking completions.  Given a parking completion $\ccc$, label each up step of $L$ with the index $i$ for which $c_i$ is its $x$-coordinate (requiring that labels of consecutive up steps increase bottom to top).  For example, the sequence $(4,6,1,2,4)$ corresponds to the decorated lattice path shown in Figure~\ref{fig:figure4}.

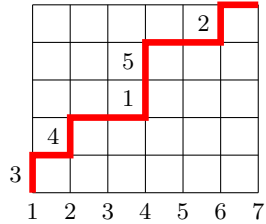
\begin{figure}[!ht]
\begin{center}
\begin{tikzpicture}[scale=0.5]
\foreach \x in {1,...,7} { \draw (\x,0)--(\x,5); \node at (\x,-.5) {\scriptsize\x}; }
\foreach \y in {0,1,...,5} { \draw (1,\y)--(7,\y);  }
\draw[line width  = 2.5pt,red,shift={(1,0)}] (0,0)--(0,1)--(1,1)--(1,2)--(3,2)--(3,4)--(5,4)--(5,5)--(6,5);
\node[left] at (1,.5) {\scriptsize 3}; 
\node[left] at (2,1.5) {\scriptsize 4}; 
\node[left] at (4,2.5) {\scriptsize 1}; 
\node[left] at (4,3.5) {\scriptsize 5}; 
\node[left] at (6,4.5) {\scriptsize 2}; 
\end{tikzpicture}
\end{center}
\caption{A decorated lattice path corresponding to $(4,6,1,2,4)$.}
\label{fig:figure4}
\end{figure}

\subsection{Operations with Lists}
In this section we set up notation for manipulation of lists, which will be necessary for studying parking completions.

\begin{definition}
If $\aaa=(a_1,\ldots,a_k)$ and $\bbb=(b_1,\ldots,b_\ell)$ are lists, then we have the following conventions:
\begin{enumerate}
\item{} $\aaa\oplus x=(a_1+x,\ldots,a_k+x)$ denotes adding the value $x$ to each entry of the list,
\item{} $\aaa\ominus x=(a_1-x,\ldots,a_k-x)$ denotes subtracting the value $x$ from each entry of the list, and
\item{} $\aaa\concat\bbb=(a_1,\ldots,a_k,b_1,\ldots,b_\ell)$ denotes the concatenation of two lists.
\end{enumerate}
\end{definition}

\begin{definition} \label{defn:shuffle}
A \emph{shuffle} of the lists $\aaa^{(1)},\dots,\aaa^{(k)}$ is a word on the multiset union of the letters of all the $\aaa^{(i)}$ with the letters coming from each $\aaa^{(i)}$ having their relative order preserved.  The set of all shuffles of $\aaa^{(1)},\dots,\aaa^{(k)}$ is denoted $\Shuffle(\aaa^{(1)},\dots,\aaa^{(k)})$.  
\end{definition}

For example, we have
\[\Shuffle((a,b),(x,y))=\{(a,b,x,y),(a,x,b,y),(a,x,y,b),(x,a,b,y),(x,a,y,b),(x,y,a,b)\}.\]
Note that
\begin{equation} \label{shuffle-inequality}
|\Shuffle(\aaa^{(1)},\dots,\aaa^{(k)})|\leq\binom{|\aaa^{(1)}|+\cdots+|\aaa^{(k)}|}{|\aaa^{(1)}|,\dots,|\aaa^{(k)}|}
\end{equation}
with equality if no letter appears in more than one of the lists $\aaa^{(i)}$.

Finally, let $\PF^*$ (respectively, $\IPF^*$) denote the set of all finite lists $\mathbf{C}=(\ccc^{(1)},\dots,\ccc^{(r)})$ of parking functions (respectively, increasing parking functions).  The sizes of the $\ccc^{(i)}$ do not have to be the same, and we also allow $\ccc^{(i)}$ to be an empty list.  In some cases, it may be convenient to regard an element of $\PF^*$ or $\IPF^*$ as an infinite list $(\ccc^{(1)},\ccc^{(2)},\dots)$, where $\ccc^{(i)}$ is the empty parking function for all $i$ sufficiently large.  The \emph{size} of $\mathbf{C}$ is $|\mathbf{C}|=\sum_i|\ccc^{(i)}|$.

\section{The join/split bijection}\label{sec:join-split}

The goal for this section is to introduce a pair of size-preserving bijections
\[
\IPF^* \begin{array}{c}
\xrightarrow{~\Join~} \\ \xleftarrow[~\Split~]{}
\end{array} \{\text{finite nondecreasing lists}\}.
\]
The right hand side corresponds to all possible nondecreasing preference sequences, of all finite lengths, for cars that could be produced.   The maps $\JOIN$ and $\SPLIT$ will be used to prove Theorem~\ref{count-IPC} and then we will carry out an appropriate ``reshuffling'' to establish Theorem~\ref{thm:main2}.

\subsection{Visualization of \Join\ and \Split} \label{sec:geom-join-split}

Recall from \S\ref{subsec:lattice-path} that every weakly increasing sequence of positive integers $\ppp=(p_1,\dots,p_n)$ corresponds to a lattice path $L:=L(\ppp)$, where the $p_i$ indicate the $x$-coordinates of up steps.  (We will follow the convention that the lattice path ends with infinitely many right steps, which we truncate as needed.)  The sequence $\ppp$ is a parking function if and only if $L$ never crosses below the line $y=x$. In other words, $\ppp$ is a parking function when $L$ is a Dyck path.  

If $\ppp$ is not a parking function, then we can look at the \emph{first violation} of $L$, namely the first time that $L$ crosses to the right of $y=x$.  The first violation will always correspond with a right step.  An example of this is shown in Figure~\ref{fig:figure5}: the dashed line is $y=x$ and the first violation is marked with a black dot.

\begin{figure}[!htb]
\begin{center}
\begin{tikzpicture}[scale=0.5]
\foreach \x in {1,...,9} { \draw (\x,0)--(\x,6); \node at (\x,-.5) {\scriptsize\x}; }
\foreach \y in {0,...,6} { \draw (1,\y)--(9,\y); }
\draw[very thick, blue, dashed] (1,0)--(7,6);
\draw[line width = 2.5pt, red] (1,0)--(1,1)--(2,1)--(2,3)--(5,3)--(5,4)--(8,4)--(8,5)--(9,5)--(9,6);
\draw[fill=black] (4.5,3) circle(.2);
\end{tikzpicture}
\end{center}
\caption{The lattice path $L(1,2,2,5,8,9)$ with the first violation of being a parking function marked.}
\label{fig:figure5}
\end{figure}
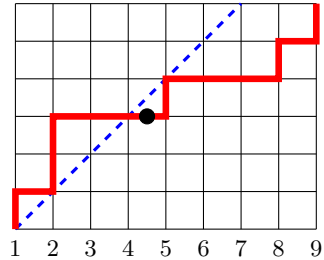

The portion of the lattice path that occurs before the first violation is a parking function.  The idea of $\SPLIT$ is to decompose a lattice path as a sequence of parking functions, connected by violating edges.  More precisely, to compute $\SPLIT(L)$, we proceed as follows:
\begin{itemize}
\item Find the longest Dyck path that is a prefix of the current lattice path.
\item Record the sequence of up steps before the first violation (which corresponds to the longest prefix that can be made into a PF when decorated as in Figure~\ref{fig:figure4}).
\item Shift coordinates so that the right endpoint of the first violation is now $(1,1)$, and redraw the line $y=x$ in the new coordinates.
\item Repeat until all up steps have been processed.
\end{itemize}
As every lattice path $L$ can be uniquely identified by a list of north steps $\ppp$, we often abuse notation and write $\SPLIT(\ppp)$, for $\SPLIT(L(\ppp))$.
The full process for applying $\SPLIT$ to $\ppp=(1,2,2,5,8,9)$ is illustrated in Figure~\ref{fig:figure6}; the result is $\Split(\ppp)=((1,2,2),(1),(),(1,2))$.  Note that in the last step we added a right step to the end in order to meet with the dashed line.

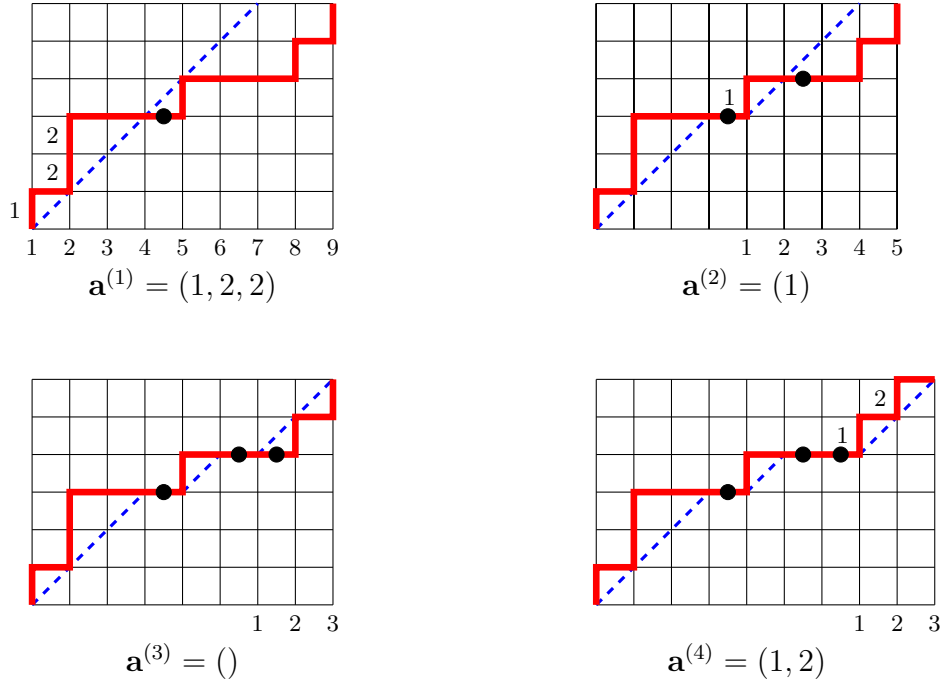
\begin{figure}[!htb]
\begin{center}
\begin{tikzpicture}[scale=0.5]
\foreach \x in {1,...,9} { \draw (\x,0)--(\x,6); \node at (\x,-.5) {\scriptsize\x}; }
\foreach \y in {0,...,6} { \draw (1,\y)--(9,\y); }
\draw[very thick, blue, dashed] (1,0)--(7,6);
\draw[line width=2.5pt, red] (1,0)--(1,1)--(2,1)--(2,3)--(5,3)--(5,4)--(8,4)--(8,5)--(9,5)--(9,6);
\node[left] at (1,.5) {\scriptsize$1$};
\node[left] at (2,1.5) {\scriptsize$2$};
\node[left] at (2,2.5) {\scriptsize$2$};
\node at (5,-1.5) {$\aaa^{(1)}=(1,2,2)$};
\draw[fill=black] (4.5,3) circle(.2); 

\begin{scope}[shift={(15,0)}]
\foreach \x/\y in {1/1,2/2,3/3,5/4,6/5,7/6} 
\foreach \x in {1,...,9} \draw (\x,0)--(\x,6);
\foreach \y in {0,...,6} \draw (1,\y)--(9,\y);
\foreach \x in {1,...,5} \node at (\x+4,-.5) {\scriptsize\x};
\draw[very thick, blue,dashed] (5,3)--(8,6) (1,0)--(4,3);
\draw[line width = 2.5pt, red] (1,0)--(1,1)--(2,1)--(2,3)--(5,3)--(5,4)--(8,4)--(8,5)--(9,5)--(9,6);
\node[left] at (5,3.5) {\scriptsize$1$};
\foreach \x/\y in {4.5/3, 6.5/4} \draw[fill=black] (\x,\y) circle(.2); 
\node at (5,-1.5) {$\aaa^{(2)}=(1)$};
\end{scope}

% Step 3
\begin{scope}[shift={(0,-10)}]
\foreach \x in {1,...,9} \draw (\x,0)--(\x,6);
\foreach \y in {0,...,6} \draw (1,\y)--(9,\y);
\foreach \x in {1,...,3} \node at (\x+6,-.5) {\scriptsize\x};
\draw[very thick, blue, dashed] 
   (5,3)--(6,4) 
   (1,0)--(4,3) 
   (7,4)--(9,6);
\draw[line width=2.5pt, red] (1,0)--(1,1)--(2,1)--(2,3)--(5,3)--(5,4)--(8,4)--(8,5)--(9,5)--(9,6);
\foreach \x/\y in {4.5/3, 6.5/4, 7.5/4} \draw[fill=black] (\x,\y) circle(.2); % violation
\node at (5,-1.5) {$\aaa^{(3)}=()$};
\end{scope}

 % Step 4
 \begin{scope}[shift={(15,-10)}]
\foreach \x in {1,...,10} \draw (\x,0)--(\x,6);
\foreach \y in {0,...,6} \draw (1,\y)--(10,\y);
\foreach \x in {1,...,3} \node at (\x+7,-.5) {\scriptsize\x};
\draw[very thick, blue, dashed]    (5,3)--(6,4)  (1,0)--(4,3) (8,4)--(10,6);
\draw[line width = 2.5pt, red] (1,0)--(1,1)--(2,1)--(2,3)--(5,3)--(5,4)--(8,4)--(8,5)--(9,5)--(9,6)--(10,6);
\node[left] at (8,4.5) {\scriptsize$1$};
\node[left] at (9,5.5) {\scriptsize$2$};
\foreach \x/\y in {4.5/3, 6.5/4, 7.5/4} \draw[fill=black] (\x,\y) circle(.2); % violation
\node at (5,-1.5) {$\aaa^{(4)}=(1,2)$};
\end{scope}
\end{tikzpicture}
\end{center}
\caption{$\Split((1,2,2,5,8,9))=((1,2,2),(1),(),(1,2))$.  All violations are marked with black dots.}
\label{fig:figure6}
\end{figure}

The $\JOIN$ operation, the reverse of $\SPLIT$, is even easier to visualize.  Given a sequence $(\aaa^{(1)},\dots,\aaa^{(r)})$ of parking functions, attach the corresponding (possibly empty) Dyck paths end-to-end in order, inserting a right step (violation) between each. Then, write down the list of $x$-coordinates of the up steps, in weakly increasing order.

\subsection{Formal Definition of \Join\ and \Split}

\begin{definition}\label{defn:Join} 
Let $(\aaa^{(1)},\dots,\aaa^{(r)})\in \IPF^*$.  For each $i\in[r]$, let $\ell_i=|\aaa^{(i)}|$; let $L_i=\ell_1+\cdots+\ell_i$, by convention $L_0=0$, and define $\bbb^{(i)}=\aaa^{(i)}\oplus(L_{i-1}+i-1)$.  Then
\[
\Join((\aaa^{(1)},\dots,\aaa^{(r)})) = \bbb^{(1)}\concat\cdots\concat\bbb^{(r)}.
\]
\end{definition}

Equivalently, $\JOIN$ can be defined recursively by
\begin{equation*} 
\Join((\aaa^{(1)},\dots,\aaa^{(r)}))=\begin{cases}
() & \text{ if } r=0,\\
\aaa^{(1)}\concat\left[\Join((\aaa^{(2)},\dots,\aaa^{(r)}))\oplus(\ell_1+1)\right] &\text{ if } r>0.
\end{cases}
\end{equation*}
Note that $\Join((\aaa^{(1)},\dots,\aaa^{(r)}))$ is an increasing sequence
because all entries in $\aaa^{(i)}$ are at most $\ell_i$ and so all entries in  $\bbb^{(i)}$ are at most $\ell_i+L_{i-1}+i-1 = L_i+i-1$; on the other hand, the entries in  $\bbb^{(i+1)}$ are at least $1+(L_i+(i+1)-1)=L_i+i+1$.

\begin{definition}\label{defn:Split}
Let $\ppp=(p_1,\dots,p_n)$ be a nondecreasing sequence.  If $|\ppp|=0$, then $\Split(\ppp)=()$.  Otherwise,  \Split\ is defined recursively as follows:
\begin{itemize}
\item[]
Let $\aaa$ be the longest prefix of $\ppp$ that is a (possibly empty) $\IPF$, and write $\ppp=\aaa\concat\qqq$.  If $\qqq$ is not empty, then $\Split(\ppp)=(\aaa,\Split(\qqq\ominus(|\aaa|+1)))$; otherwise, $\SPLIT(\ppp)=\aaa$.
\end{itemize}
\end{definition}

Alternatively, \Split\ can be computed by the following iterative algorithm:

\begin{minipage}{\textwidth}
\begin{framed}
\noindent\textbf{Initialize:} $\ppp^{(0)}$ := $\ppp$; \ $L_0:= 0$; \ $r=0$

\noindent\textbf{Loop:} while $|\ppp^{(r)}|>0$

\indent\indent\indent\indent
\newlength{\savetabcolsep}
\setlength{\savetabcolsep}{\tabcolsep}
\setlength{\tabcolsep}{3pt}
\begin{tabular}{lll}
$r$&:= &$r+1$\\
$\aaa^{(r)}$ & := & the longest initial sequence of $\ppp^{({r-1})}$ that is an $\IPF$\\
$\qqq^{(r)}$ & := & suffix of $\ppp^{({r-1})}$ following $\aaa^{(r)}$ (so that $\ppp^{({r-1})}=\aaa^{(r)}\concat\qqq^{(r)}$)\\
$L_r$ &:= & $L_{r-1}+|\aaa^{(r)}|$\\
$\ppp^{(r)}$ & := & $\qqq^{(r)}\ominus(|\aaa^{(r)}|+1)$
\end{tabular}

\setlength{\tabcolsep}{\savetabcolsep}
\noindent\textbf{Output:} $(\aaa^{(1)},\dots,\aaa^{(r)})$
\end{framed}
\end{minipage}

\begin{example}
Let $\ppp=\ppp^{(0)}=(1,2,2,5,8,9)$.  The $\SPLIT$ algorithm proceeds as follows:
\[\begin{array}{l | llllllllllllll}
r && \aaa^{(r)} && \qqq^{(r)}  && L_r && \ppp^{(r)}\\ \hline
1 && (1,2,2) && (5,8,9) && 3 && (5,8,9)\ominus4 = (1,4,5)\\
2 && (1) && (4,5) && 4 && (4,5)\ominus2 = (2,3)\\
3 && () && (2,3) && 4 && (2,3)\ominus1 = (1,2) \\
4 && (1,2) && ()  && 6 && ()
\end{array}\]
Thus, $\Split(\ppp)=( (1,2,2), (1), (), (1,2) )$.

If we take this output, $( (1,2,2), (1),(), (1,2))\in \IPF^*$, and put it into $\JOIN$, then we have the following: $(\ell_1,\ell_2,\ell_3,\ell_4)=(3,1,0,2)$, $(L_1,L_2,L_3,L_4)=(3,4,4,6)$, and
\begin{align*}
&\hspace{-25pt}\Join( (1,2,2), (1),(), (1,2))\\
&= [\aaa^{(1)}\oplus({L_0+0})]\concat[\aaa^{(2)}\oplus({L_1+1})]\concat[\aaa^{(3)}\oplus({L_2+2})]\concat[\aaa^{(4)}\oplus({L_3+3})]\\
&=[ (1,2,2)\oplus0 ]\concat[ (1)\oplus4] \concat [()\oplus6] \concat[ (1,2)\oplus7]\\
&=(1,2,2)\concat(5)\concat()\concat(8,9)
= (1,2,2,5,8,9) = \ppp.
\end{align*}
\end{example}

\begin{theorem}
The functions $\JOIN$ and $\SPLIT$ are mutual inverses, hence bijections.
\end{theorem}
\begin{proof}
If $\ppp$ is an $\IPF$, then $\SPLIT(\ppp)=(\ppp)$ and $\JOIN((\ppp))=\ppp$.  Suppose that $\ppp$ is not an $\IPF$, then $\ppp = \aaa\concat\qqq$ where $\aaa$ is the maximal initial subsequence that is an $\IPF$.  Under the map $\SPLIT$, this will be sent to $(\aaa,\Split(\qqq\ominus(|\aaa|+1))$.  Conversely, given a list of increasing parking functions of the form $(\aaa)\concat \AAA$ under the map $\JOIN$, this will be sent to $\aaa\concat[\Join(\AAA)\oplus(|\aaa|+1)]$.  In particular, the two maps preserve the first initial $\IPF$.  Applying induction using $\qqq\ominus(|\aaa|+1)$ and $\AAA$ establishes the result.
\end{proof}

\section{Enumerating (increasing) parking completions}\label{sec:enumeration-proofs}

We now use the $\JOIN$ and $\SPLIT$ maps from  \S\ref{sec:join-split} to establish the enumerative results.

Throughout this section we let $\ttt=(t_1,\ldots,t_m)$ be the taken spots, and $\uuu=(u_1,\ldots,u_{n-m})$ be the unoccupied spots of $[n]$.

\begin{lemma} \label{u-inequalities}
We have the following:
\begin{enumerate}
\item For all $i$, we have $u_{t_i-i}<t_i<u_{t_i-i+1}$.
\item For all $i,j$ we have  $j\geq t_i-i+1$ if and only if $u_j\geq j+i$.
\end{enumerate}
\end{lemma}
\begin{proof}
We know that $t_i-i$ is the number of unoccupied spots before $t_i$.  So the $(t_i-i)$-th unoccupied spot will occur to the left of $t_i$, and the $(t_i-i+1)$-th unoccupied spot will occur to the right of $t_i$.  This establishes assertion~(1).

Assertion~(1) can be restated as $u_{t_i-i}-(t_i-i)<i$ and $u_{t_i-i+1}-(t_i-i+1)\geq i$.  Since the sequence $(u_1-1,u_2-2,\dots)$ is weakly increasing, it follows that $u_{j}-j\geq i$ if and only if $j\ge t_i-i+1$.  This establishes assertion~(2).
\end{proof}

We want to relate an increasing parking completion with taken spots~$\ttt$ with a list $\AAA=(\aaa^{(1)},\dots)$ of increasing parking functions.  Informally, the idea is to ``fill in'' the gaps between the taken spots (if too many cars are trying to squeeze into one gap, they can move to the next).  The first $i$ increasing parking functions should park $L_i$ cars to fill at least the first $i$ gaps (and possibly additional spots), which translates into the condition $L_i\ge t_i-i$.  We formalize this with the following definition.

\begin{definition}
Let $\ttt=(t_1,\ldots,t_m)$ with $t_m\le n$, $\AAA=(\aaa^{(1)},\dots,\aaa^{(m+1)})\in \IPF^*$, and let $L_i=|\aaa^{(1)}|+\cdots+|\aaa^{(i)}|$.  Then $\AAA$ is \emph{$(n,\ttt)$-compatible} if $L_i\ge t_i-i$ for all $i\in [m]$ and $L_{m+1}=n-m$.  We write $\IPF^*_{n,\ttt}$ for the set of all $(n,\ttt)$-compatible elements of $\IPF^*$.
\end{definition}

\begin{example}
Let $\AAA=((1,1,2,2),(1),(1),())$.  Then $\AAA\in\IPF^*_{9,(3,6,7)}$ and $\AAA\in\IPF^*_{10,(2,5,7,9)}$, but $\AAA\not\in\IPF^*_{9,(6,7,8)}$.
\end{example}

We now establish the connection between $\IPF^*_{n,\ttt}$ and $\IPC$.

\begin{theorem}\label{thm:ipf-ipc}
Let $\ttt=(t_1,\dots,t_m)$ with $t_m\leq n$.  Then we have
$\SPLIT(\IPC_n(\ttt))=\IPF^*_{n,\ttt}$ and $\JOIN(\IPF^*_{n,\ttt})=\IPC_n(\ttt)$.
\end{theorem}

\begin{proof}
First, given $\AAA=(\aaa^{(1)},\dots,\aaa^{(m+1)})\in\IPF^*_{(n,\ttt)}$, we need to show that $\Join(\AAA)=(c_1,\ldots,c_{n-m})\in\IPC_n(\ttt)$.  It is equivalent to show that $c_k\le u_k$ for all $k\in[n-m]$.  Suppose that $c_k$ arises from $\aaa^{(i+1)}_j$ (the $j$-th entry of $\aaa^{(i+1)}$). Then $k=L_{i}+j$ and $c_k=L_{i}+i+\aaa^{(i+1)}_j$.  Since $\aaa^{(i+1)}$ is an increasing parking function, it follows that $\aaa^{(i+1)}_j\le j$.  So we have $c_k\le L_i+i+j=k+i\le u_k$. The last inequality follows from Lemma~\ref{u-inequalities}(2) since $k=L_i+j\ge (t_i-i)+j\ge t_i-i+1$ and $\AAA$ is $(n,\ttt)$-compatible.

Second, given $\ppp=(p_1,\dots,p_{n-m})\in\IPC_n(\ttt)$, we need to show that $\AAA=(\aaa^{(1)},\ldots,\aaa^{(m+1)})=\SPLIT(\ppp)\in\IPF^*_{(n,\ttt)}$.  Since $\AAA\in\IPF^*$ we need to show that $L_{m+1}=n-m$ (which follows by noting that $|\AAA|=|\ppp|=n-m$) and that $L_i\ge t_i-i$ for $i\in[m]$.  Suppose, for the sake of contradiction, that $L_i+i<\ttt_i$ for some $i$.  The first term in $\aaa^{(i+1)}$ will always be $1$.  Looking at the term in $\ppp$ that the first term produces we can conclude $p_{L_i+1}=L_i+i+1$.  Since $\ppp$ will be a weakly increasing sequence, all other terms from that point forward in the output from $\JOIN$ will be at least as large; showing that none of those terms can park in the spots at or before $L_i+i$.  That leaves $L_i$ cars to fill the open spots which are in $[L+i]$.  This is impossible since there are at most $(i-1)$ taken spots which have occurred on or before $L_i+i$ and so at least $L_i+1$ spots needing to be filled.
\end{proof}

\subsection{Enumeration Formula for Increasing Parking Completions (Proof of Theorem~\ref{count-IPC})}
By Theorem~\ref{thm:ipf-ipc} there is a bijection between $\IPC_n(\ttt)$ and $\IPF^*_{n,\ttt}$. To enumerate $\IPF^*_{n,\ttt}$, we group the elements of $\IPF^*_{n,\ttt}$ according to the sequence $(|\aaa^{(1)}|,\ldots,|\aaa^{(m+1)}|)=(\ell_1,\ldots,\ell_{m+1})$; note that in all cases $\ell_1+\cdots+\ell_j\ge t_j-j$.  Each $\aaa^{(i)}$ corresponds to an increasing parking function of length $i$, for which there are $C_{\ell_i}=\frac1{\ell_i+1}\binom{2\ell_i}{\ell_i}$ possibilities.  Putting this all together we have
\[
\big|\IPC_n(\ttt)\big| = \big|\IPF_{n,\ttt}^\ast\big|= \sum_{\ell\in L_n(\ttt)} \prod_{i=1}^{m+1} \frac{1}{\ell_i+1} \binom{2\ell_i}{\ell_i},
\]
where
\[
L_n(\ttt)= \left\{\bl=(\ell_1,\dots,\ell_{m+1})\in\Nn^{m+1}\,\middle\vert
\begin{array}{l}
\ell_1+\cdots+\ell_j \geq t_j-j \text{ for all } j\in[m], \text{ and}\\
 \ell_1+\cdots+\ell_{m+1}=n-m 
 \end{array} \right\}.
 \]

\subsection{Enumeration Formula for Parking Completions (Proof of Theorem \ref{thm:main2})}\label{sec:proof:main2}

We now enumerate ordinary (i.e., not necessarily increasing) parking completions.  In what follows, hatted symbols (such as $\hat{\aaa}$) indicate sequences not required to be in increasing order; unhatted symbols (such as $\aaa$) indicate weakly increasing sequences.  Given $\widehat{\ppp}\in \PC_n(\ttt)$, let $\ppp\in \IPC_n(\ttt)$ be its (unique) weakly increasing rearrangement.  Let $\Split(\ppp)=(\aaa^{(1)},\dots,\aaa^{(m+1)})\in\IPF^*$, so that
\[
\ppp=\JOIN(\SPLIT(\ppp))=\bbb^{(1)}\concat\cdots\concat\bbb^{(m+1)}
\]
where $\bbb^{(i)}=\aaa^{(i)}\oplus(L_{i-1}+(i-1))$ and $L_i=\ell_1+\cdots+\ell_i$ with $\ell_i=|\aaa^{(i)}|$.  Moreover, the sequences $\bbb^{(j)}$ are ``setwise increasing'': if $i<j$, then every entry in $\bbb^{(i)}$ is less than every entry in $\bbb^{(j)}$.  Therefore, given $\bl=(\ell_1,\dots,\ell_{m+1})$, the $\bbb^{(i)}$, and therefore the $\aaa^{(i)}$, can be recovered from any shuffle of $\ppp$.

We now classify (increasing) parking completions by the length sequences of their splits.  For $\bl=(\ell_1,\dots,\ell_{m+1})$, define
\[
\widehat{\mathcal{S}}_{n,\ttt}(\bl)=\big\{
(\widehat\aaa^{(1)}\oplus(L_0+0))\concat\cdots\concat(\widehat\aaa^{(m+1)}\oplus(L_m+m)) \mid 
\widehat\aaa^{(i)}\in\PF\text{ and }|\widehat\aaa^{(i)}|=\ell_i\text{ for all $i$}\big\}.
\]
In particular,
\[
|\widehat{\mathcal{S}}_{n,\ttt}(\bl)|=\prod_{j=1}^{m+1}(\ell_j+1)^{(\ell_j-1)}.
\]

Now consider the set $\Shuffle(\widehat{\mathcal{S}}_{n,\ttt}(\bl))$ (see Definition~\ref{defn:shuffle}).  As noted above, the $m+1$ sequences $\aaa^{(i)}\oplus(L_{i-1}+(i-1))$ can be recovered individually from any shuffle of them.  Therefore, equality holds in~\eqref{shuffle-inequality}:
\[
|\Shuffle(\widehat{\mathcal{S}}_{n,\ttt}(\bl))|=\binom{n-m}{\bl}|\widehat{\mathcal{S}}_{n,\ttt}(\bl)|=\binom{n-m}{\bl}\prod_{j=1}^{m+1}(\ell_j+1)^{(\ell_j-1)}.
\]
Summing over all possible $\bl$ gives the final formula for the number of parking completions:
\[
\big|\pco{n}{\ttt}\big|=\sum_{\bl\in L_n(\ttt)} \binom{n-m}{\bl}\prod_{j=1}^{m+1}(\ell_j+1)^{(\ell_j-1)},
\]
where
\[L_n(\ttt)= \left\{\bl=(\ell_1,\dots,\ell_{m+1})\in\Nn^{m+1}\,\middle\vert
\begin{array}{l}
\ell_1+\cdots+\ell_j \geq t_j-j \text{ for all } j\in[m], \text{ and}\\
\ell_1+\cdots+\ell_{m+1}=n-m 
\end{array} \right\}.\]

\subsection{Parking Completions of a Block (Proof of Corollary~\ref{thm:main1})}
We now specialize to the case that the taken spots consist of a contiguous block, i.e., $\ttt=(i+1,\ldots,i+m)$.  Parking completions with a single spot taken ($m=1$, $i$ arbitrary) were enumerated by Diaconis and Hicks~\cite[Corollary 1]{diaconis}, who showed (in our notation) that
\begin{equation} \label{eqn:DH}
\big|\pco{n}{(k)}\big|=\sum_{s=0}^{n-k}\binom{n-1}{s}(s+1)^{s-1}(n-s)^{n-s-2}.
\end{equation}
The case $i=0$ with arbitrary $m$ was first considered by Yan~\cite{yan-GPF} under the name of \emph{$c$-parking functions}.  Gessel and Seo~\cite[\S10]{GesselSeo} showed that for $\ell\leq n$,
\begin{equation} \label{eqn:EH}
\big|\pco{n}{(1,\ldots,\ell)}\big|=(\ell+1)(n+1)^{n-\ell-1}.
\end{equation}
This formula was generalized by Ehrenborg and Happ~\cite[Theorem 1.2]{ehrenborg2} to the case of cars of different sizes.  Our Corollary~\ref{thm:main1}, which we now prove, covers the case that both $i$ and $m$ are arbitrary, generalizing both~\eqref{eqn:DH} and~\eqref{eqn:EH}.

Let $\widehat\ccc\in\pco{n}{(i+1,\ldots,i+m)}$.  Then as in the discussion in \S\ref{sec:proof:main2} we have
\[
\widehat\ccc=\widehat{\aaa}\concat(\widehat\ppp\oplus(|\widehat{\aaa}|+1)),
\]
where $\widehat\aaa$ is a parking function and $\widehat\ppp$ is a parking completion.  For the remainder of our discussion we set $k=|\widehat\aaa|$.

Because $\widehat{\ccc}$ contains a violation immediately after $\widehat\aaa$, it \emph{must} be the case that the first $k$ cars (those whose preferences occur in $\widehat{\aaa}$) fill at least the first $i$ spots (and possibly more).  Hence $i\le k\le n-m$ (the upper bound comes when $\widehat\aaa$ would park in all available spaces).  For a fixed choice of $k$ there are $(k+1)^{k-1}$ possible choices for the parking function $\widehat\aaa$.

After using the preferences from $\widehat\aaa$, the first $k+m$ spots will be filled.  This means that $\widehat\ppp$ is a parking completion of $[n-k-1]$ where the first $m-1$ spots are already taken (after appropriate shifting).  For a fixed $k$, formula~\eqref{eqn:EH} implies that there are $m(n-k)^{n-k-m-1}$ such parking completions.

We now have counts for $\widehat\aaa$ and $\widehat\ppp$ for a particular choice of $k$.  The last thing to note is that we can combine them together by shuffling, and since the entries are distinct this introduces the binomial coefficient $\binom{n-m}{k}$.  So adding over all possible choices of $k$ we have
\[
|\pco{n}{(i+1,\ldots,i+m)}\big|=
\sum_{k=i}^{n-m}\binom{n-m}{k}(k+1)^{k-1}m(n-k)^{n-k-m-1}.
\]
This establishes Corollary~\ref{thm:main1}.

If we now sum over all of the possible $\ttt$ consisting of $m$ continuous entries, the result simplifies nicely.

\begin{proposition}\label{prop:all_shifts}
Let $1\le m\le n$.  Then
$\displaystyle\sum_{i=0}^{n-m}|\pco{n}{(i+1,\ldots,i+m)}\big| = (n+1)^{n-m}$.
\end{proposition}
\begin{proof}
We will do a variation of the classic proof that there are $(n+1)^{n-1}$ parking functions.  

Place the spots $1,\ldots,n$ in a circle and add a new spot labeled $0$.  Each car drives to its preferred spot and then travels, with wraparound, until it parks.  Therefore, in this setting \emph{all cars park}.  We first select $\ttt$ which can be done in $n+1$ ways, and this fixes the first $m$ locations of the cars. The remaining $n-m$ cars now select an arbitrary location and park, this can be done in $(n+1)^{n-m}$ ways.  Finally we note that in each case exactly one spot remains open, and by shifting all preferences each spot is equally likely to be open over all combinations (so each spot has probability $1/(n+1)$ of being open).  A parking completion corresponds to when $0$ is the open spot and this happens $(n+1)^{n-m}$ times.
\end{proof}

It is possible to more directly derive the result of Corollary~\ref{thm:main1} from Theorem~\ref{thm:main2}. Observe that setting $\ell_1=k$ and $\ttt=(i+1,\dots,i+m)$, we can rewrite Theorem~\ref{thm:main2} as
\[
\big|\pco{n}{\ttt}\big|=\sum_{k=i}^{n-m}\binom{n-m}{k} (k+1)^{k-1} \sum_{(\ell_2,\dots,\ell_{m+1})\compn n-m-k} \binom{n-m-k}{\ell_2,\dots,\ell_{m+1}} \prod_{j=2}^{m+1}(\ell_j+1)^{(\ell_j-1)}.
\]
Therefore, Corollary~\ref{thm:main1} will follow from the following lemma, which gives a combinatorial interpretation of the inner sum.

\begin{lemma}
For all fixed $k$ satisfying $1\le k\le n-m$, we have
\[
\sum_{(\ell_2,\ldots,\ell_{m+1})\vDash n-m-k}\binom{n-m-k}{\ell_2,\ldots ,\ell_{m+1}}\prod_{j=2}^{m+1}(\ell_j+1)^{\ell_j-1} = m(n-k)^{n-k-m-1}
\]
where $(\ell_2,\ldots,\ell_{m+1})\vDash n-m-k$ is a composition of $n-m-k$.
\end{lemma}

\begin{proof}
Fix $\ell_2,\ldots,\ell_{m+1}$.  The left-hand side of the desired equality is the number of trees on vertices $1,\dots,n-k+1$ that can be constructed by the following procedure.
\begin{enumerate}
\item Start with $m$ buckets labeled $2,\dots,m+1$.
\item Put each element of $\{m+2,\ldots,n-k+1\}$ into one of the buckets.  (The buckets are allowed to be empty.)  Let $B_i$ be the set of elements in the $i$-th bucket and let $\ell_i=|B_i|$.
\item For each $i\in\{2,\dots,m+1\}$, choose a tree $T_i$ with vertices $B_i\cup\{i\}$.
\item Add a new vertex 1 adjacent to each of $2,\dots,m+1$.
\end{enumerate}
Steps (1) and (4) involve no choice, while there are $\binom{n-m-k}{\ell_2,\ldots ,\ell_{m+1}}$ possibilities for step (2) and $\prod_{j=2}^{m+1}(\ell_j+1)^{\ell_j-1}$ possibilities for step (3).

The right-hand side counts the Pr\"{u}fer codes for labeled trees of the form described (where the Pr\"ufer code is constructed by iteratively deleting the \textit{largest} leaf vertex and recording its neighbor).  The last $m$ entries of the code correspond to deleting vertices $2,\dots,m+1$, and are all $1$. The $(m+1)$-th entry from the end must be one of $2,\dots,m+1$, so there are $m$ choices. The remaining $n-k-m-1$ entries can be any number except $1$, so there are $n-k$ choices for each of them. This gives $m(n-k)^{n-k-m-1}$ Pr\"{u}fer codes, finishing the proof.
\end{proof}

\section{Connections to Other Parking Function Variations}\label{sec:enumerative-connections}

\subsection{\texorpdfstring{$\uuu$}{u}-Parking Functions and Pitman-Stanley Polytopes}
Let $\uuu=(u_1,\ldots, u_N)\in\Nn^N$ with $u_1\leq\cdots\leq u_N$.  A \emph{$\uuu$-parking function} \cite[\S1.4.1]{yan} is a sequence $\widehat{\aaa}\in\Nn^N$ whose non-decreasing rearrangement $\aaa=(a_1,\ldots,a_N)$ satisfies $1\leq a_i\leq u_i$ for all $i\in[N]$.\footnote{In \cite{yan}, which uses zero-indexing, this condition is written as $0\leq a'_i<u_i$.}  The set of all $\uuu$-parking functions is denoted $\PF(\uuu)$.  Parking completions are a special case of $\uuu$-parking functions: if $u_1<\cdots<u_N\le n$ and $\ttt$ is the strictly increasing complement of $\uuu$ in $[n]$, then $\PF(\uuu)=\pco{n}{\ttt}$.  

Let $\xxx=\Delta\uuu=(u_1,u_2-u_1,\dots,u_N-u_{N-1})$.  Pitman and Stanley \cite[Theorem 1]{PitmanStanley}, \cite[Theorems 1.27 and 1.29]{yan} showed that the number of $\uuu$-parking functions is
\begin{equation} \label{StanleyPitmanFormula}
P_N(\xxx)=|\PF(\uuu)|=N!\sum_{\kkk\in K_N} \prod_{i=1}^N \frac{x_i^{k_i}}{k_i!} = \sum_{\kkk\in K_N}\binom{N}{\kkk} x_1^{k_1}\cdots x_N^{k_N}
\end{equation}
where $K_N$ is the set of \emph{balanced sequences of length $N$}, i.e.,
\[K_N=\left\{\kkk\in\Nn^N\,\middle\vert\,\sum_{i=1}^j k_i\geq j\text{ for all } j\in[N-1] \text{ and } \sum_{i=1}^N k_i=N\right\}.\]

The number $P_N(\xxx)/N!$ is the volume of the so-called \emph{Pitman-Stanley polytope}
\[
\Pi_N(\xxx) = \Big\{\yyy\in\Rr^N\mid y_i\geq 0 \text{ and } y_1+\cdots+y_i\leq x_1+\cdots+x_i \text{ for all } i\in[N]\Big\}.
\]

In addition, there is a determinantal formula for $|\PF(\uuu)|$~\cite[Theorem~1.25]{yan}, which can be obtained using the theory of Gon\v{c}arov polynomials:
\begin{equation} \label{det-upf}
|\PF(\uuu)| = \det\left(s_{ij}\right)_{i,j=1}^N, \quad\text{where}\quad
s_{ij}=\begin{cases}\dfrac{u_i^{j-i+1}}{(j-i+1)!} &\text{ if } j+i-1\geq 0,\\[10pt] 0 &\text{ otherwise.} \end{cases}
\end{equation}

The index set and summation formula in the Pitman-Stanley formula~\eqref{StanleyPitmanFormula} strongly resemble those of
Theorem~\ref{thm:main2}.  However, the following example suggests that there does not seem to be a simple translation between the two.

\begin{example} \label{compare-PS-and-our-formula}
Let $n=4$, $\uuu=(1,4)$, and $\ttt=(2,3)$.  Then
\[
\pco{n}{\ttt}=\PF(\uuu)=\{(1,1), (1,2), (1,3), (1,4), (2,1), (3,1), (4,1)\}.
\]
In Theorem~\ref{thm:main2}, the index set for the summation is
\begin{align*}
L_n(\ttt) &= \left\{\bl=(\ell_1,\ell_2,\ell_3)\in\Nn^3\mid\ell_1\geq 1,\ \ell_1+\ell_2\geq 1,\ \ell_1+\ell_2+\ell_3=2\right\}\\
&= \{(1,1,0),\ (1,0,1),\ (2,0,0)\},
\end{align*}
so Theorem~\ref{thm:main2} gives
\begin{align*}
|\pco{n}{\ttt}| &=
\binom{2}{1,1,0} 2^0 2^0 1^{-1} +
\binom{2}{1,0,1} 2^0 1^{-1} 1^{-1} +
\binom{2}{2,0,0} 3^1 1^{-1} 1^{-1}\\
&= 2+2+3 = 7.
\end{align*}

On the other hand, setting $N=|\uuu|=2$ and $\xxx=\Delta\uuu=(1,3)$, the balanced sequences of length~$N$ are
\begin{align*}
K_2 &= \left\{\kkk=(k_1,k_2)\in\Nn^2\mid k_1\geq 1,\ k_1+k_2=2\right\}\\
&= \{(1,1),\ (2,0)\},
\end{align*}
so the Pitman-Stanley formula~\eqref{StanleyPitmanFormula} gives
\begin{align*}
P_N(\xxx) =|\PF(\uuu)|
&= \binom{2}{1,1} 1^1 3^1 ~+~ \binom{2}{2,0} 1^2 3^0\\
&= 6+1 = 7.
\end{align*}
Neither of the compositions $6+1$ and $3+2+2$ of $|\PF(\uuu)|=7$ refines the other, making it unlikely that either~\eqref{StanleyPitmanFormula} or Theorem~\ref{thm:main2} can be obtained from the other.
\end{example}

\begin{question}
The summands in~\eqref{StanleyPitmanFormula} give the volumes of the pieces in a natural decomposition of the associated Pitman-Stanley polytope \cite[Theorems~9 and~18]{PitmanStanley}.  Does there exist an analogous decomposition realizing Theorem~\ref{thm:main2} geometrically?
\end{question}

Such a decomposition would be essentially different from that given by Pitman and Stanley, and would be of substantial value in understanding these polytopes.

\subsection{Signature Parking Functions}
Recall from Section~\ref{sec:geom-join-split} that $L(\ppp)$ denotes the lattice path with up-steps at $x$-coordinates given by the entries of a weakly increasing integer sequence $\ppp$.  If $\sss=(s_1,\dots,s_a)$ is a composition (i.e., an arbitrary sequence of positive integers), we define $\tilde L(\sss)=L(s_1,s_1+s_2-1,\dots,\sum_{i=1}^js_i-(j-1),\dots,)$.  As an example, the lattice path for $\sss=(3,4,4,2,1)$ is shown in Figure~\ref{fig:sigpath} in red.

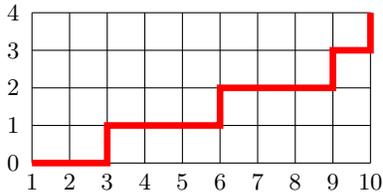
\begin{figure}[htb]
\centering
\begin{tikzpicture}[scale=0.5]
\foreach \x in {1,...,10} { \draw (\x,0)--(\x,4); \node at (\x,-.5) {\scriptsize\x}; }
\foreach \y in {0,...,4} { \draw (1,\y)--(10,\y); \node at (.5,\y) {\scriptsize\y}; }
\draw[line width = 2.5pt, red] (1,0)--(3,0)--(3,1)--(6,1)--(6,2)--(9,2)--(9,3)--(10,3)--(10,4);
\end{tikzpicture}

\caption{The lattice path corresponding to the composition $\sss=(3,4,4,2,1)$.}
\label{fig:sigpath}
\end{figure}

\textit{Signature Dyck paths} were introduced by Ceballos and Gonz\'alez~D'Le\'{o}n~\cite{CeballosGonzalez}
as a generalization of \emph{rational Dyck paths} (see, e.g., \cite{ALW,ARW,GMV1,GMV2}).  We have modified their definition slightly by expressing the signature as a lattice path, but it is simple to translate between the two settings.

\begin{definition}\cite[Defn.~3.2]{CeballosGonzalez}
A \emph{Dyck path with signature $\sss$}, or for short an \emph{$\sss$-Dyck path}, is a lattice path $D$ consisting of right and up steps starting at $(1,1)$ and ending at $(b,a)=(|\sss|-\ell(\sss)+1,\ell(\sss))$ which is weakly above and to the left of the lattice path $\tilde{L}(\sss)$.  The collection of all Dyck paths with signature $\sss$ is denoted $\mathcal{DP}_\sss$.
\end{definition}

A \emph{signature parking function (with signature $\sss$)} then corresponds to a labeling of the up steps of an $\sss$-Dyck path, just as in the correspondence between parking functions and decorated lattice paths (see Figure~\ref{fig:figure4}).
  
Given a signature $\sss$, define a partition $\lambda=\lambda(\sss) = (s_1-1,s_1+s_2-2,\ldots,s_1+\cdots+s_a-a)$ (with parts listed in weakly increasing order).  Then the lattice paths contained in the Ferrers diagram of~$\lambda$ are precisely those in $\mathcal{DP}_\sss$.  In general, signature Dyck paths with mild restrictions on the entries and increasing parking completions are equivalent.

\begin{proposition}\label{dyck-signature-equivalence}
Let $\sss=(s_1,\ldots,s_a)$ be a composition with $s_1,s_a\ge 1$ and $s_i\ge 2$ for all $1<i<a$.  Let $n = 1+s_1+\cdots+s_a-a$, let $\uuu = (s_1,s_1+s_2-1,\ldots,s_1+\cdots+s_{a-1}-(a-2))$, and let $\ttt$ consist of the ordered list of the elements in $[n]$ not in $\uuu$.  Then $|\mathcal{DP}_{\sss}|=\IPC_n(\ttt)$.
\end{proposition}

For example, if $\sss=(2,3)$ then $\lambda(\sss)=(1,3)$, and the $\sss$-Dyck paths are those shown in Figure~\ref{fig:figure2}.  In general, given lists $\ttt,\uuu$ of taken and unoccupied spots in $[n]$, we can reconstruct a signature $\sss=(s_1,\ldots,s_a)$ of length $a=|\uuu|+1$ by setting $s_i=u_i-u_{i-1}-1$ (with the conventions $u_0=0$ and $u_{a}=n+1$).

Consequently, Theorems~\ref{thm:main2} and~\ref{count-IPC} can easily be translated into enumeration formulas for signature parking functions and signature Dyck paths, respectively.

\bibliography{bibliography}{}
\bibliographystyle{plain}
\end{document}